\documentclass[12pt]{amsart}

\usepackage{amssymb}

\usepackage[all]{xy}

\usepackage{enumitem}

\makeatletter
\@namedef{subjclassname@2020}{%
  \textup{2020} Mathematics Subject Classification}
\makeatother

\usepackage[T1]{fontenc}

\newtheorem{theorem}{Theorem}[section]

\newtheorem{lemma}[theorem]{Lemma}

\theoremstyle{definition}
\newtheorem{definition}[theorem]{Definition}
\newtheorem{remark}[theorem]{Remark}

\newtheorem{question}[theorem]{Question}

\numberwithin{equation}{section}

\DeclareMathOperator{\clop}{clop} 
\DeclareMathOperator{\iden}{id} 
\DeclareMathOperator{\mS}{\mathcal{S}} 
\DeclareMathOperator{\V}{\mathcal{V}}

\DeclareMathOperator{\B}{\mathcal{B}} 
\DeclareMathOperator{\cov}{cov} 
 
\DeclareMathOperator{\lc}{lc} 
\DeclareMathOperator{\fix}{fix} 
\DeclareMathOperator{\ext}{ext} 
\DeclareMathOperator{\Pow}{Pow}

\newcommand{\Ws}{\mathbf{W}^*}
\newcommand{\aN}{\alpha \mathbb{N}}
\newcommand{\bN}{\beta \mathbb{N}}
\newcommand{\N}{\mathbb{N}}

\newcommand{\Comp}{\mathbf{Comp}}

\newcommand{\coC}{\mathbf{coComp}}
\newcommand{\hoWs}{\mathbf{hoW}^{*}}
\newcommand{\hoW}{\mathbf{hoW}}

\begin{document}


\baselineskip=17pt


\title[Minimum spaces and vector lattices]{Some minimum topological spaces, and vector lattices}

\author{Ricardo E. Carrera}
\address{Department of Mathematics \\
Nova Southeastern University \\
Fort Lauderdale, FL 33328 \\ USA}
\email{ricardo@nova.edu}

\author{Anthony W. Hager}
\address{Department of Mathematics and Computer Science \\
Wesleyan University \\ Middletown, CT 06459 \\ USA}
\email{email: ahager@wesleyan.edu}

\author[B. Wynne]{Brian Wynne$^*$}
\address{Department of Mathematics \\
 Lehman College, CUNY \\
 Gillet Hall, Room 211 \\
 250 Bedford Park Blvd West \\
 Bronx, NY 10468 \\
 USA}
\email{brian.wynne@lehman.cuny.edu}

\date{}

\begin{abstract}
We investigate the existence of compact Hausdorff spaces $X$ that are minimum with respect to $cX=K$ for some fixed covering operator $c$ and compact Hausdorff space $K$ with $cK=K$. Then, using the Yosida representation theorem, we show how that situation relates to the existence of Archimedean vector lattices $A$ with distinguished strong unit that are  minimum with respect to $hA=H$ for some fixed hull operator $h$ and vector lattice $H$ with $hH=H$. Among others, we obtain answers for $c=g$ (the Gleason covering operator), $c=qF$ (the quasi-$F$ covering operator), $h = u$ (the uniform completion operator), and $h=e$ (the essential completion operator).
\end{abstract}

\subjclass[2020]{Primary 54D80, 06F20; Secondary 54C10, 08C05}

\keywords{compact Hausdorff space, covering operator, vector lattice, hull operator}

\maketitle

\section{Introduction}

Let $\Comp$ be the category of compact Hausdorff spaces with continuous maps. For any $\Comp$-object $K$ (we usually write $K \in \Comp$ for short), the set of all pairs $(X,f)$ where $X \xleftarrow{f} K$ is a cover in $\Comp$ is (modulo an equivalence relation) a complete upper semilattice with respect to a natural partial order. That partial order is inherited by $\{ X \in \Comp : cX=K \}$, for any $K \in \Comp$ and any covering operator $c$ on $\Comp$ (such a $c$ assigns to each $X \in \Comp$ a cover $X \xleftarrow{c_X} cX$). One of the two main questions addressed in this paper is:  

\begin{question}\label{Q1}
For what covering operators $c$ on $\Comp$, and for what $K \in \Comp$ with $cK=K$, does $\{ X \in \Comp : cX=K \}$ have the minimum, and in that case what is it?
\end{question} 

After giving some preliminary information on covering operators and hull operators in Sections \ref{prelim1}-\ref{YosidaSec}, we establish answers to \ref{Q1} in Section \ref{Q1answers} for $c = \iden$ (the identity covering operator), $c=a(\gamma)$ (an atom in the lattice of covering operators), $c=g$ (the Gleason covering operator), and $c=qF$ (the quasi-$F$ covering operator). With the help of the Yosida representation theorem, Sections \ref{BAstuff} and \ref{Q2answers} transfer the results of Section \ref{Q1answers} into answers to a related question for Boolean algebras and to our second main question, which concerns the category $\Ws$ of Archimedean vector lattice with distinguished strong unit and unit-preserving vector lattice homomorphisms:

\begin{question}\label{Q2}
For what hull operators $h$ on $\Ws$, and for what $H \in \Ws$ with $hH=H$, does $\{ A \in \Ws : hA=H \}$ have the minimum, and in that case what is it?
\end{question}




\section{Preliminaries on covering operators}\label{prelim1}

We give here a quick review of covers and covering operators on $\Comp$. Our main references are \cite{H89}, \cite{PW88}, and \cite{CH11}.

All spaces are assumed to be Tychonoff, and are usually compact Hausdorff. If $X$ is a space, then a pair $(Y,f)$ is called a \emph{cover} of $X$ if $Y$ is a space and $X \xleftarrow{f} Y$  is a perfect irreducible continuous surjection; the map $f$ itself is called a \emph{covering map}. Here, \emph{perfect} means that the map is closed with all point-inverses compact, and \emph{irreducible} means that $f(F) \neq X$ for any proper closed set $F$ in $Y$. 

For two covers $(X_i,f_i)$ ($i=1,2$) of $K \in \Comp$, we write $(X_1, f_1) \leq (X_2,f_2)$ if there is a continuous $X_1 \xleftarrow{f} X_2$ such that $f_2 = f_1 \circ f$; in that case $f$ is unique and $(X_2,f)$ is a cover of $X_1$. If both $(X_1,f_1) \leq (X_2, f_2)$ and $(X_2, f_2) \leq (X_1,f_1)$, then the $f$ associated with the first inequality is a homeomorphism. The preorder $\leq$ induces an equivalence relation on the collection of all covers of $K$; we denote that collection $\cov K$. Usually, we abuse notation and identify a cover $(X,f)$ of $K$ with the equivalence class in $\cov K$ that it represents. It turns out that the relation $\leq$ makes $\cov K$ into a complete lattice. Indeed, that $\cov K$ has a top element is a consequence of the following fact, which is due to Gleason \cite{G58}.

\begin{theorem}\label{Gleason}
If $E \in \Comp$, then $\cov E = \{ E \}$ if and only if $E$ is extremally disconnected (iff $X$ is projective in $\Comp$).  Moreover, each $X \in \Comp$ has a maximum cover $(gX,g) \in \cov X$, called the Gleason (or projective) cover.
\end{theorem}

Let $c$ be a function assigning to each $X \in \Comp$ a $(cX,c_X) \in \cov X$, and let $\fix(c) = \{ X \in \Comp : c_X \mbox{ is a homeomorphism} \}$. Such a $c$ is a \emph{covering operator} on $\Comp$  if $(cX,c_X)$ is the minimum in $\{ (Y,f) \in \cov X: Y \in \fix(c) \}$; in that case we call $\fix(c)$ the \emph{covering class} associated with $c$ . For example, the $g$ in \ref{Gleason} may be viewed as such a covering operator, whose associated covering class is the class of all extremally disconnected spaces in $\Comp$; we call $g$ the Gleason covering operator.

Let $\coC$ be the (proper) class of all covering operators on $\Comp$. For the reader familiar with  category theory, the elements of $\coC$ are precisely the coreflections in $\Comp^{\#}$, where $\Comp^{\#}$ is the category with the same object class as $\Comp$ but whose morphisms are only the covers in $\Comp$. One may partially order $\coC$ by letting $c\leq d$ just in case $(cX,c_X) \leq (dX,d_X)$ for every $X \in \Comp$; and with this order $\coC$ becomes a complete lattice with bottom the identity operator and top the Gleason operator $g$. Other members of $\coC$ will be discussed as they arise in Section \ref{Q1answers}.

In addition to partially ordering the covers above a fixed space, one may also partially order the spaces covered by a fixed space in a similar way: if $X_i \xleftarrow{f_i} K$ ($i=1,2$) are covers in $\Comp$, then we write $(X_1,f_1) \leq (X_2,f_2)$ just in case there is  a continuous $X_1 \xleftarrow{f} X_2$ such that $f_1 = f \circ f_2$. This $\leq$ is also a preorder that leads to a partial order on the set of equivalence classes determined by the preorder (one may establish this as in \cite[Theorem 8.4(e)]{PW88}). In fact, with that order, the set of equivalence classes forms a complete upper semilattice (\cite[Theorem 8.4(f)]{PW88}). 

To simplify the notation, we make the following

\begin{definition}\label{S(K,c)defn}
Suppose $c \in \coC$ and $K \in \Comp$ with $K = cK$. Let 
\[
\mS (K,c) \equiv \{ X \in \Comp : cX = K \}.
\]
\end{definition}

Thus \ref{Q1} may now be expressed by asking about the minimum in $\mS(K,c)$, where the order is  inherited from the order described above on the collection of all spaces of which $K$ is a cover (here we are identifying the space $X$ with the pair $(X,c_X)$).


\section{Preliminaries on hull operators}\label{prelim2}

Here we give a brief review of essential embeddings and hull operators. This is in some regards dual to what was discussed in Section \ref{prelim1}. We assume some familiarity with these concepts; our main references are \cite{LZ71} (for vector lattices), \cite{HS79} (for essential embeddings), and \cite{CH11} (for hull operators) .

If $A \in \Ws$, then a pair $(B,f)$ is called an \emph{essential extension} of $A$ if $B \in \Ws$ and $A \xrightarrow{f} B$ is an \emph{essential embedding}, i.e., if $f$ is a $\Ws$-embedding such that whenever $B \xrightarrow{f_1} C$ is a $\Ws$-morphism and $f_1 \circ f$ is monic, then $f_1$ is monic (in $\mathbf{W}^*$, monic means injective). Equivalently, a $\Ws$-embedding $A \xrightarrow{f} B$ is essential if for every $0 < b \in B$ there is $0 < a \in A$ such that $f(a) \leq nb$ for some $n \in \N$ (this is sometime referred to as $A$ being ``large" in $B$, see \cite[Theoreme 11.1.15]{BKW77}).

For two essential extensions $(B_i,f_i)$ $(i=1,2)$ of $A \in \Ws$, we write $(B_1, f_1) \leq (B_2, f_2)$ if there is a $\Ws$-morphism $B_1 \xrightarrow{f} B_2$ such that $f_2 = f \circ f_1$; in that case $f$ is unique and $(B_2,f)$ is an essential extension of $B_1$. If both $(B_1,f_2) \leq (B_2,f_2)$ and $(B_2,f_2) \leq (B_1,f_1)$, then the $f$ associated with the first inequality is a $\Ws$-isomorphism. The preorder $\leq$ induces an equivalence relation on the collection of all essential extensions of $A$; we denote that collection $\ext A$. It turns out that the relation $\leq$ makes $\ext A$ into a complete lattice. Indeed, that $\ext A$ has a top element is a consequence of the following fact, which is due to Conrad \cite[Theorem 3.3]{C71}.

\begin{theorem}\label{essclosure}
If $A \in \Ws$, then $\ext A = \{ A \}$ if an only if $A = C(X)$ for some extremally disconnected $X \in \Comp$. 
Moreover, each $A \in \Ws$ has a maximum essential extension $(eA,e_A) \in \ext A$, called the essential completion of $A$.
\end{theorem}

Note: throughout this paper $C(X)$ denotes the vector lattice of all real-valued continuous functions on the space $X$; if $K \in \Comp$, then $C(K) \in \Ws$, where we implicitly take the distinguished strong unit to be $\chi(X)$, i.e., the constant function on $X$ with value 1.

Let $h$ be a function assigning to each $A \in \Ws$ a pair $(hA,h_a) \in \ext A$, and let $\fix(h) = \{ A \in \Ws : h_A \mbox{ is a $\Ws$-isomorphism}\}$. Such an $h$ is a \emph{hull opertor} on $\Ws$ if $(hA,h_A)$ is the minimum in $\{ (B,f) \in \ext A : B \in \fix(h) \}$; in that case we call $\fix(h)$ the \emph{hull class} associated with $h$. For example, the $e$ in \ref{essclosure} may be viewed as such a hull operator, whose associated hull class is the class of all $C(X)$ with $X \in \Comp$ extremally disconnected; we call $e$ the essential completion operator. 


Let $\hoWs$ be the (proper) class of all hull operators on $\Ws$. For the reader familiar with category theory, the elements of $\hoWs$ are precisely the reflections in $\bf{W}^{*\#}$, where $\bf{W}^{*\#}$, is the category with the same object class as $\Ws$ but whose morphisms are only the essential embeddings in $\Ws$. One may partially order $\hoWs$ by letting $h \leq k$ just in case $(hA,h_A) \leq (kA,k_A)$ in $\ext A$ for every $A \in \Ws$; with this order, $\hoWs$ becomes a complete lattice with bottom the identity operator and top the essential closure operator $e$. Other members of $\hoWs$ will be discussed as they arise in Section \ref{Q2answers}.

In addition to partially ordering the essential extensions of a fixed $\Ws$-object, one may also partially order the $\Ws$-objects that have a common essential extension in a similar way: if $A_i \xrightarrow{f_i} B$ $(i=1,2)$ are essential $\Ws$-embeddings, then we write $(A_1, f_1) \leq (A_2,f_2)$ just in case there is an essential $\Ws$-embedding $A_1 \xrightarrow{f} A_2$ such that $f_1 = f_2 \circ f$. This $\leq$ is also a preorder that leads to a partial order on the set of equivalence classes determined by the preorder (one may establish this as in before \ref{S(K,c)defn}). In fact, with that order, the set of equivalence classes forms a complete upper semilattice.

To simplify notation, we make the following

\begin{definition}
Suppose $h \in \hoWs$ and $H \in \Ws$ with $H = hH$. Let 
\[
\V (H,h) \equiv \{ A \in \Ws : hA = H \}.
\]
\end{definition}

Thus \ref{Q2} may now be expressed by asking about the minimum in $\V(H,h)$, where the order is inherited from the order described above on the collection of all $\Ws$-objects of which $H$ is an essential extension (here we are identifying the $\Ws$-object $A$ with the pair $(A,h_A)$).


\section{The Yosida functor and $\mu$}\label{YosidaSec}

The connection between $\Ws$ and $\Comp$ that we exploit here arises from the following representation theorem, which developed from a result of Yosida \cite{Y42}. 

\begin{theorem}\label{yosida}
\begin{itemize}
\item[(i)] (The Yosida space) If $A \in \Ws$ with distinguished strong unit $u$, then $YA \in\Comp$, where $YA$ is the set of all $u$-maximal ideals of $A$ (ideals maximal with respect to not containing $u$) with the hull-kernel topology. We call $YA$ the \emph{Yosida space} of $A$.
\item[(ii)] (Existence) There is a $\Ws$-embedding $\nu_A \colon A \to C(YA)$ such that $\nu_A(A)$ separates the points of $YA$.
\item[(iii)] (Uniqueness) If $X \in \Comp$ and $\phi \colon A \to C(X)$ is a $\Ws$-embedding such that $\phi(A)$ separates the points of $X$, then there is a homeomorphism $\eta \colon YA \to X$ such that $\nu_A(a) = \phi(a) \circ \eta$ for every $a \in A$.
\item[(iv)] (Covariance) If $A \xrightarrow{f} B$ is a $\Ws$-morphism, then there is a continuous surjection $YG \xleftarrow{Yf} YH$ such that $\nu_B(f(a)) = \nu_A(a) \circ Yf$ for every $a \in A$. Moreover, $(B,f) \in \ext A$ if and only if  $(YH,Yf) \in \cov YG$. 
\end{itemize}
\end{theorem}

\begin{remark}
Theorem \ref{yosida} is actually a special case of a representation theorem for $\mathbf{W}$, the category of Archimedean vector lattices with distinguished weak unit and unit-preserving vector lattice homomorphisms; see \cite{HR77}. 
\end{remark}

In \cite[Section 5]{CH11}, a map $\mu$ from covering classes in $\Comp$ to hull classes of Archimedean lattice-ordered groups with weak unit is studied. Here we recast $\mu$ so that it yields hull operators in $\Ws$.

\begin{definition}
For $c \in \coC$ and $A \in \Ws$, let
$\mu(c)A = C(cYA)$ and $\mu(c)_A = \phi_{A,c} \circ \nu_A$, where $\nu_A$ is the Yosida embedding and $\phi_{A,c} \colon C(YA) \to C(cYA)$ is the essential $\Ws$-morphism induced by the cover $YA \xleftarrow{c_{YA}} cYA$, i.e., the map given by $\phi_{A,c}(f) = f \circ c_{YA}$ for $f \in C(YA)$. 
\end{definition}

\begin{theorem}
If $c \in \coC$, then $\mu(c) \in \hoWs$.
\end{theorem}

\begin{proof}
First, $(\mu(c)A,\mu(c)_A) \in \ext A$ because both $\nu_A$ and $\phi_{A,c}$ are essential $\Ws$-embeddings (this follows from \ref{yosida}; note that $c_{YA} = Y\phi_{A,c}$), and a composition of essential $\Ws$-embeddings is an essential $\Ws$-embedding. 

Now, suppose $(B,f) \in \ext A$ and $B \xrightarrow{\mu(c)_B} \mu(c)B$ is a $\Ws$-isomorphism. We show that $(\mu(c)A,\mu(c)_A) \leq (B,f)$ in $\ext A$. Since a composition of covers is a cover, we see that $(cYB, Yf \circ c_{YB}) \in \cov YA$. Since $c  \in \coC$, it follows that $(cYA, c_{YA}) \leq (cYB, Yf \circ c_{YB})$ in $\cov YA$. So there is a cover $cYA \xleftarrow{Y\varphi} cYB$ that induces an essential $\Ws$-embedding $C(cYA) \xrightarrow{\varphi} C(cYB)$ such that the diagrams
\begin{displaymath}
\xymatrix{
  YA  & YB \ar[l]_{Yf} \\
  cYA \ar[u]^{Yc_{YA}} &  cYB \ar[l]^{Y\varphi}  \ar[u]_{Yc_{YB}}   }
\end{displaymath}
and
\begin{displaymath}
\xymatrix{
  A  \ar[d]_{\mu(c)_A}  \ar[r]^{f}  & B  \ar[d]^{\mu(c)_B} \\
 C(cYA) \ar[r]_{\varphi} &  C(cYB)    }
\end{displaymath}
 both commute. Thus $\mu(c)_B^{-1} \circ \varphi$ witnesses that $(\mu(c)A,\mu(c)_A) \leq (B,f)$ in $\ext A$.
\end{proof}

What follows are some basic properties of $\mu$. 

\begin{theorem}\label{mustuff}
\begin{itemize}
\item[(i)] $\mu$ is order-preserving and injective.
\item[(ii)] The range of $\mu$ is contained in $\{ h \in \hoWs : h \mbox{ is above } u \}$.
\item[(iii)] $\mu$ is not onto $\hoWs$: omitted are the ring-reflection $r$ and the projectable hull operator $p$. 
\end{itemize}
\end{theorem}

\begin{proof}
(i) Evident.

(ii) By (i) and \ref{mu(id)=u} ($\iden$ below $c$ implies $\mu(\iden)$ below $\mu(c)$).

(iii) For a space $X$, let $F(X) \equiv \{ f \in C(X) : f(X) \mbox{ is finite} \}$. For any infinite zero-dimensional $X$, $F(X) = rF(X) = pF(X)$, but $F(X)$ is not $uF(X)$. (Using $h$ below $k$ if and only if $h(\Ws)$ contains $k(\Ws)$, which we know).  See \cite[Section 6]{HR77} for background on $r$; see, for example, \cite{HW24} and the references given there for background on $p$ (for $\ell$-groups).
\end{proof}

\begin{remark}
Actually, in \ref{mustuff}(ii) above, we have equality. This is not informative to the sequel here, and a proof seems tedious, and involves a map $\sigma$ of $\hoWs$, unfortunately not quite ranging in $\coC$, namely $\sigma(h)X \equiv YhC(X)$. This can fail to be idempotent, e.g., $\sigma(p)$ is not.

Some of the complications are observable in \cite{CH11}, where $\sigma$ is defined on a subclass of $\hoW$ where $\sigma$ is idempotent, and $\mu$ is later presented as a ``section for $\sigma$".

We omit further discussion here, intending to return to the $\mu$/$\sigma$ situation in a later paper.
\end{remark}

We end this section by identifying some particular images of $\mu$.

\begin{theorem}\label{mu(g)=e}
$\mu(g) = e$.
\end{theorem}

\begin{proof}
By definition $\mu(g)A =C(gYA)$ and $\mu(g)_A = \phi_{A,g} \circ \nu_A = \nu_A$. Since $gYA$ is extermally disconnected, it follows from \ref{essclosure} that $C(gYA) = eA$ and $\mu(g)_A = e_A$.
\end{proof}

Let $u \in \hoWs$ be the uniform completion operator, i.e., the operator for which $uA = C(YA)$ and $u_A = \nu_A$. The name of this operator refers to the fact that the uniform completion of $\nu_A(A)$ in $C(YA)$ is $C(YA)$ by the Stone-Weierstrass theorem since $\nu_A(A)$ is a vector lattice in $C(YA)$ that contains all constant functions and separates the points of $YA$.

\begin{theorem}\label{mu(id)=u}
$\mu(\iden) = u$.
\end{theorem}

\begin{proof}
By definition $\mu(\iden)A =C(\iden YA) = C(YA)$ and $\mu(\iden)_A = \phi_{A,\iden} \circ \nu_A = \nu_A$.
\end{proof}


\section{Some answers to question \ref{Q1}}\label{Q1answers}

In this section we look at $\mS(K,c)$ for the following cases: $c = \iden$ (the identity covering operator), $c=a(\gamma)$ (an atom in $\coC$), $c=g$ (the Gleason covering operator), and $c=qF$ (the quasi-$F$ covering operator).

First, we get a triviality out of the way:

\begin{theorem}
If $\iden \in  \coC$ is the identity operator, then $\mS(K,\iden) = \{ K \}$ for any $K \in \Comp$, so $K$ is the minimum.
\end{theorem}

In \cite{HW21}, it is shown that the lattice $\coC$ has atoms, the description of which we briefly recall now. Suppose $E \in \Comp$ is extremally disconnected and $p \neq q$ are non-isolated points in $E$. Let $E_{pq}$ be the quotient of $E$ that identifies just $p$ and $q$, and let $E_{pq} \xleftarrow{\gamma} E$ be the associated quotient map. Changing the notation, let $E_{\gamma}$ denote $E_{pq}- \{\gamma(p) \}$. Then $E_{pq} = \dot{E}_{\gamma}$ is the one-point compactification of $E_{\gamma}$. Thus, one has the cover $\dot{E}_{\gamma} \xleftarrow{\gamma} E$. The atom $a(\gamma) \in \coC$ determined by $\gamma$ is then given by
\[
(a(\gamma)Y, a(\gamma)_Y) = \left \{
\begin{array}{ll}
(E, \gamma) & \mbox{ if } Y = \dot{E}_{\gamma} \\
(Y, \iden_Y) & \mbox{ if } Y \neq \dot{E}_{\gamma}.
\end{array}
\right.
\] 
(Here $Y=\dot{E}_{\gamma}$ means those spaces are homeomorphic, and $\iden_Y$ is the identity function on the space $Y$.) In \cite[Theorem 2.2]{HW21}, it is shown that $a(\gamma)$ is a an atom of $\coC$, and that every atom in $\coC$ is of that form. Note that the class of spaces fixed by $a(\gamma)$ consists of all spaces not homeomorphic to $\dot{E}_{\gamma}$. We will refer to $a(\gamma)$ as an atom of $\coC$ associated with $E$.

\begin{theorem}\label{a(gamma)}
\begin{itemize}
\item[(i)] $\mS(E,a(\gamma)) = \{ \dot{E}_{\gamma}, E \}$, which has minimum $\dot{E}_{\gamma}$.
\item[(ii)] If $Y \neq E, \dot{E}_{\gamma}$, then $\mS(Y,a(\gamma)) = \{ Y \}$, which has minimum $Y$.
\end{itemize}
\end{theorem}

\begin{proof}
These follow immediately from the above description of $a(\gamma)$.
\end{proof}

Next, let $g \in \coC$ be the Gleason cover, so that for $K \in \Comp$ one has $gK = K$ if and only if $K$ is extremally disconnected. Of course, if $K$ is finite, then $K$ is discrete, $gK=K$, and $\mS(K,g) = \{ K \}$ has minimum $K$. So we turn to the case when $K$ is infinite. 

If $D$ is an infinite discrete space, then we write $\alpha D$ for the one-point compactification of $D$.

\begin{lemma}\label{Dlemma}
\begin{itemize}
\item[(i)] If $D$ is an infinite discrete space, and if $X \xleftarrow{\tau} \alpha D$ is a cover, then $\tau$ is a homeomorphism.
\item[(ii)] If $D$ is an infinite discrete space, then the covers of $\alpha D$ are exactly the compactifications of $D$.
\item[(iii)] If $K \in \Comp$ is infinite, then $\{ Z \in \Comp : K \in \cov Z \}$ has a minimum if and only if $K$ is a compactification of some discrete space $D$, and in that case the minimum is $\alpha D$.
\end{itemize}
\end{lemma}

\begin{proof}
One can check that, for $X,Y \in \Comp$, if $X \xleftarrow{\tau} Y$ is any cover, then $\tau(p)$ is isolated in $X$ for any isolated $p \in Y$ and $\tau^{-1}(q) = \{ r \}$ for some isolated $r \in Y$ whenever $q \in X$ is isolated. Thus $\tau$ restricts to a homeomorphism between the isolated points of $Y$ and the isolated points of $X$. Statements (i) and (ii) follow easily from this observation. 

If $K$ is a compactification of the discrete space $D$, then $D$ is infinite because $K$ is infinite. Using the observation in the previous paragraph again, one may show that the partially ordered set of compactifications of $D$ (see, e.g., \cite[Section 3.5]{E89}) is order-isomorphic to the partially ordered set of covers of $\alpha D$. The sufficiency in statement (iii) follows. 

Finally, suppose $K \in \Comp$ is not the compactifaction of any discrete space $D$. Let $Z \xleftarrow{f} K$ be a cover. We show that $Z$ is not the minimum. Since $f$ is a cover, and since the set of isolated points of $K$ is not dense, the set of isolated points of $Z$ is not dense. So there is a non-empty open set $U$ in $Z$ containing no isolated points. Choose points $p \neq q$ in $U$, and let $Z_{pq}$ be the quotient of $Z$ obtained by identifying just $p$ and $q$. Then $Z_{pq} \in \Comp$, the quotient map $Z_{pq} \leftarrow Z$ is a cover and not a homeomorphism, and $K \in \cov Z_{pq}$ (any composition of covers is a cover). Thus $Z$ is not the minimum.
\end{proof}

\begin{remark}
The proof of \ref{Dlemma}(iii) shows that, in addition to having no minimum, $\{ Z \in \mathbf{Comp} : K \in \cov Z \}$ has no minimal elements when $K$ is not a compactification of some discrete space.
\end{remark}

If $X$ is any Tychonoff space, then $\beta X$ denotes the \v{C}ech-Stone compactification of $X$ (see, e.g.,  \cite[Section 3.6]{E89} or \cite[Chapter 6]{GJ60}). 

\begin{theorem}\label{c=g}
If $K \in \Comp$ is infinite with $gK=K$, then $\mS(K,g)$ has the minimum if and only if $K = \beta D$ for some discrete space $D$, and in that case the minimum is $\alpha D$. 
\end{theorem}

\begin{proof}
Suppose $K \in \Comp$ is infinite with $gK = K$. Then $\mS(K,g)=\{ Z \in \Comp : Z \in \cov K \}$, so $\mS(K,g)$ has the minimum if and only if $K$ is a compactification of an infinite discrete space $D$ (recall that $K$ is infinite by hypothesis) and the minimum is $\alpha D$ by \ref{Dlemma}. Since $K$ is extremally disconnected, $D$ is $C^*$-embedded in $K$ (\cite[Problem 1H.6]{GJ60}). Thus $K = \beta D$.
\end{proof}

Finally, we consider the case when the covering operator in question is $c=qF$, the quasi-$F$ covering operator. The covering class associated with this operator is the class of quasi-$F$ spaces, which are the Tychonoff spaces $X$ such that every dense cozero-set of $X$ is $C^*$-emebedded (every bounded real-valued continuous function on a dense cozero-set extends continuously over $X$). See \cite{DHH80} or \cite{CH11} for more about $qF$ and the quasi-$F$ spaces.  

Recall that a space $X$ is called \emph{almost-$P$} if $X$ has no proper dense cozero-set (see \cite{L77}). Observe that almost-$P$ implies quasi-$F$.

\begin{theorem}\label{c=qF}
\begin{itemize}
\item[(i)] $\mS(\bN,qF)$ has minimum $\aN$.
\item[(ii)] If $K \in \Comp$ is almost-$P$, then $\mS(K,qF) = \{ K \}$ has minimum $K$. 
\end{itemize}
\end{theorem}

\begin{proof}
Since $\mathbb{N}$ is infinite, discrete, and extremally disconnected, we know from \ref{Dlemma} that $\{ Z \in \Comp : \bN \in \cov Z \}$ coincides with the set compactifications of $\mathbb{N}$ and has minimum $\aN$. Since 
\[
\mS(\bN,qF) \subseteq \{ Z \in \Comp : \bN \in \cov Z \},
\]
and since $qF\aN = \bN$ (every dense open set of $\aN$ contains a dense cozero-set of $\aN$, so $qF \aN$ is extremally disconnected by \cite[Proposition 4.6]{DHH80}), statement (i) follows.

Every almost-$P$ space is also quasi-$F$, and if $X \leftarrow Y$ is a cover and $Y$ is almost-$P$, then $X$ is almost-$P$ too.
So if $K \in \Comp$ is almost-$P$, it follows that $\mS(K,qF) = \{ K \}$. Thus (ii) holds.
\end{proof}

\begin{remark}
Every extremally disconnected space is quasi-$F$, and $\mS(\beta D,g)$ has minimum $\alpha D$ by \ref{c=g}. On the other hand, if $D$ is uncountable, then $\mS(\beta D,qF)$ is what? In any event, $\alpha D \notin \mS(\beta D, qF)$ because $\alpha D$ is almost-$P$ (hence quasi-$F$). 
\end{remark}


\section{Boolean Algebras}\label{BAstuff}

We pose and answer the following question about minimum Boolean algberas for the completion operator $\chi$; it's a variant (indeed, a special case) of Question \ref{Q2}.

For $C$ a Boolean algebra with $\chi C = C$, let
\[
\B(C,\chi) = \{ A \in \mathbf{BA} : \chi A = C \},
\]
where $\bf{BA}$ is the category of Boolean algebras and Boolean algebra homomorphisms.
This $\B(C,\chi)$ is partially ordered by $\mathbf{BA}$-inclusion.

\begin{question}\label{QBA}
For what $C \in \mathbf{BA}$ with $\chi C=C$, does $\B(C,\chi)$ have a minimum, and then what is it?
\end{question}

Here is our answer.

\begin{theorem}\label{BAmin}
$\B(C, \chi)$ has the minimum if and only if $C$ is a power set $\Pow(D)$, and then the minimum is the finite/cofinite Boolean algebra on $D$.
\end{theorem}

We obtain this via Stone Duality, relating the question to the special case of Q\ref{Q1} embodied in \ref{c=g}.

Stone Duality ``is" the pair of functors $\mathbf{BA} \xrightarrow{s} \mathbf{ZD} \xrightarrow{\clop} \mathbf{BA}$, where $\mathbf{ZD}$ is the category of compact zero-dimensional spaces with continuous maps, $s$ assigns a Boolean algebra $A$ to its Stone representation space $sA$, and $\clop X$ is the $\bf{BA}$ of clopen sets in $X$. $\clop \circ s$ is a $\bf{BA}$-isomorphism, and we generally identify $A \in \bf{BA}$ with $\clop sA$. Upon so doing, all $\bf{BA}$-morphisms $\clop X \xrightarrow{\varphi} \clop Y$ are given from a unique continuous $X \xleftarrow{\tau} Y$, as $\varphi(U) = \tau^{-1}(U)$. We denote $\varphi$ as $\widehat{\tau}$. (The similarity with the Yosida representation in \ref{yosida} is evident.) For background on Boolean algebras, see \cite{S69} as needed.

Now, for $A \in \bf{BA}$, $\chi A$ is the maximum $\bf{BA}$-essential extension of $A$: for a $\bf{BA}$-morphism $A \xrightarrow{\varphi} B$ being $\bf{BA}$-essential means $\varphi(A)$ is order-dense in $B$, and $\clop X \xleftarrow{\widehat{\tau}} \clop Y$ is $\bf{ZD}$-essential if and only if $\tau$ is irreducible (originally due to E. Weinberg; cf. the comment about $\Ws$ \ref{yosida}(iv)).

It results that a completion $A \xrightarrow{\chi_A} \chi A$ is exactly $\clop sA \xleftarrow{\widehat{\tau}} \clop gsA$, where $sA \xleftarrow{\tau} gsA$ is the Gleason cover.

\begin{theorem}\label{n.3}
For any complete $C \in \bf{BA}$, $C = \clop E$, where $E$ is compact and extremally disconnected, and we have an order-isomorphism of posets $\B(C,\chi) \to \mS(E,g) \cap \bf{ZD}$ given as: $\chi A = C$, for $A = \clop X$ means just $X \in \bf{ZD}$, and $gX = E$ (see the discussion above)

Thus, a minimum $M$ in $\B(\clop E, \chi)$ corresponds exactly to a minimum in $\mS(E,g) \cap \bf{ZD}$.
\end{theorem}

We recall:

\begin{theorem}\label{c=g2}
(\ref{c=g} above) For $E$ infinite and extremally disconnected, $\mS(E,g)$ has the minimum if and only if $E$ is a $\beta D$, and then the minimum is $\alpha D$.
\end{theorem}

Now we give the proof of \ref{BAmin}.

\begin{proof}
Equivalent to \ref{BAmin} is the statement: If $E$ is extremally disconnected, then $\B(\clop E, \chi)$ has the minimum if and only if $E$ is a $\beta D$, and then the minimum is $\clop \alpha D$. By \ref{n.3}, this last is equivalent to: $\mS(E,g) \cap \bf{ZD}$ has the minimum if and only if $E$ is a $\beta D$, and then the minimum is $\alpha D$. We prove this.

$(\Leftarrow)$. Suppose $E$ is a $\beta D$, with minimum $\alpha D$. Apply \ref{c=g2}; since $\alpha D \in \bf{ZD}$, it is also a minimum in $\mS(E,g) \cap \bf{ZD}$.

$(\Rightarrow)$. Suppose $\mS(E,g) \cap \bf{ZD}$ has the minimum $Z$. It suffices to show $E$ is a $\beta D$. For then by \ref{c=g2} again, $\mS(E,g)$ has the minimum $\alpha D$.  To show $E$ is a $\beta D$, it suffices to show $E$ is a compactification of a discrete space $D$, which then must be $\beta D$ (since $E$ is extremally disconnected, and thus any dense set is $C^*$-embedded). To show that, repeat the argument at the end of the proof of \ref{Dlemma} as: if $E$ is not a compactification of a $D$, then there is clopen $U$ containing non-isolated $p \neq q$, and $Z_{pq} \xleftarrow{\gamma} Z$ which identifies only $p$ and $q$ is a cover, and it's easy to see that $Z_{pq}$ is $\bf{ZD}$ because $Z$ was. That contradicts our assumption that $Z$ was the minimum in $\mS(E,g) \cap \bf{ZD}$.
\end{proof}

\begin{remark}
\ref{BAmin} is distinctly related to \ref{F(aD)} in the next section, via the categorical embedding $\mathbf{BA} \to \Ws$, which sends $\clop X$ to the $\ell$-group in $C(X)$ generated by all the characteristic functions of the $U \in \clop X$. Explaining that seems to require all the above details anyway. 
\end{remark}


\section{Some answers to question \ref{Q2}}\label{Q2answers}

In this section we establish a connection between $\V(H,\mu(c))$ and $\mS(K,c)$, then use it to transfer our results from the previous section.

We first give a result about minimums relative to a fixed Yosida space.
Recall that if $K \in \Comp$, then the set $\{ A \in \Ws : YA = K \}$ may be partially ordered by writing $A \leq B$ if there is an essential vector lattice embedding $A \xrightarrow{f} B$ such that $\nu_A = \nu_B \circ f$ (here $\nu_A, \nu_B$ are the Yosida embeddings of \ref{yosida}(ii)).

\begin{theorem}\label{Ymin}
If $K \in \Comp$, then the poset $\{ A \in \Ws : YA = K \}$ has the minimum if and only if $K$ is zero-dimensional, and in that case the minimum is $F(K) \equiv \{ f \in C(K) : f(K) \mbox{ is finite} \}$.
\end{theorem}

\begin{proof}
($\Leftarrow$): Suppose $K$ is zero-dimensional and $A \in \Ws$ has $YA = K$. Then $A$ contains the constant function with value 1 on $YA$, and since $A$ is vector lattice, it must therefore contain all constant functions on $YA$. Moreover, since $A$ 0-1 separates closed sets in $YA$ (this follows from \ref{yosida}), one sees that $\chi(U) \in A$ for any $U \in \clop(YA)$, where $\chi(U)$ denotes the characteristic function of $U$. It follows that $F(K) \subseteq A$. Now $YF(K)=K$ because $K$ is zero-dimensional (so $F(K)$ separates the points of $K$), so $F(K)$ is indeed the desired minimum.

($\Rightarrow$): First, we establish some simple lemmas.

For $A \in \Ws$, $p \in YA$, and $a \in A$, we say that $a$ is \emph{locally constant at $p$} if there is a neighborhood $U$ of $p$ for which $a \vert_U$, the restriction of $a$ to $U$, is constant. Let $\lc(A,p)$ be the set of all $a \in A$ such that $a$ is locally constant at $p$.

\begin{lemma}\label{lem0}
$\lc(A,p) \in \Ws$ and $Y(\lc(A,p)) = YA$.
\end{lemma}

\begin{proof}
For $\otimes = +, -, \vee, \wedge$ and $a_1, a_2 \in \lc(A,p)$, with associated  neighborhoods $U_i$ of $p$, one sees that the neighborhood $U_1 \cap U_2$ of $p$ witnesses that $a_1 \otimes a_2 \in \lc(A,p)$.

Then, $\lc(A,p)$ separates points of $YA$, by cases of $p_1 \neq p_2$.
\end{proof}

\begin{lemma}\label{lem1}
If $a$ is locally constant at every $p \in YA$, then $a \in F(YA)$.
\end{lemma}

\begin{proof}
For every $p \in YA$, take a neighborhood $U_p$ of $p$ such that  $a$ takes the value $a(p)$ everywhere on $U_p$. Take $\{ U_1, \dots, U_k \}$ a finite subcover of $\{ U_p : p \in YA \}$. So $a(YA) = \bigcup \{ a(p_i) : i = 1, \dots, k  \}$ is finite.
\end{proof}

\begin{lemma}\label{lem2}
Suppose $A \in \Ws$. $A = F(YA)$ iff  $A = \lc(A,p)$ for every $p \in YA$, and in that case $YA$ is ZD.
\end{lemma}

\begin{proof}
($\Rightarrow$): obvious.

($\Leftarrow$): By \ref{lem1}, every $a \in A$ is also in $F(YA)$.
\end{proof}

We are now ready to complete the proof of \ref{Ymin}. Suppose $K$ is not zero-dimensional and $A \in \Ws$ with $YA = K$. Then, by \ref{lem2}, we have $A \neq F(K)$ and there is $p \in K$ such that $A \neq \lc(A,p)$. But $\lc(A,p) \subseteq A$ by definition, and $Y\lc(A,p)=K$ by \ref{lem0}. Thus $A$ is not the minimum.
\end{proof}

Here is a connection between $\V$ and $\mS$.

\begin{theorem}\label{VtoS}
$M$ is minimum in $\V(H,\mu(c))$ if and only if $YM$ is minimum in $\mS(YH,c)$, $YM$ is zero-dimensional, and $M=F(YM)$.
\end{theorem}

\begin{proof}

Keep in mind the following:
\begin{itemize}
\item[(i)] $A \in \V(H,\mu(c))$ if and only if $cYA = YH$ (and $cYH=YH$);
\item[(ii)] $A \in \V(H,\mu(c))$ and  $YA' = YA$ implies $A' \in \V(H,\mu(c))$.
\end{itemize}

($\Rightarrow$) Suppose $\V(H,\mu(c))$ has minimum $M$ with $YM=X_0$. Then $cX_0 = YH$ by (i), so $X_0 \in \mS(YH,c)$. Now, if $X \in \mS(c,YH)$, then $X \leftarrow cX = YH$,  so $C(X) \in \V(H,\mu(c))$ by (i) since $YC(X)=X$. Thus there is an essential $\Ws$-embedding $M \rightarrow C(X)$ by the minimality of $M$, so $X_0 = YM \leftarrow YC(X) = X$ and $X_0$ is minimum in $\mS(c,YH)$. Finally, take $A \in \Ws$ with $YA = X_0$. Then $A \in \V(H,\mu(c))$ by (ii), so there is an essential $\Ws$-embedding $M \rightarrow A$ by the minimality of $M$. Hence $YM$ is zero-dimensional and $M=F(YM)$ by \ref{Ymin}.

($\Leftarrow$): Suppose $YM$ is minimum in $\mS(YH,c)$, $YM$ is zero-dimensional, and $M=F(YM)$. Then $M$ is minimum in $\{ A \in \Ws : YA = YM \}$ by \ref{Ymin}, and $M \in \V(H,\mu(c))$ by (i). Now if $A \in \V(H,\mu(c))$, then $YA \in \mS(YH,c)$ by (i), so $YM \leftarrow YA$ by the minimality of $YM$.  Thus there is an essential $\Ws$-embedding $M \rightarrow A$ and $M$ is minimum in $\V(H,\mu(c))$.
\end{proof}

Now we combine \ref{VtoS} with our results on $\mS(K,c)$ for various $c$ from the previous section.

\begin{theorem}\label{u}
 If $K \in \Comp$, and if $u \in \hoWs$ is the uniform completion operator, then $\V(C(K),u)$ has a minimum if and only if $K$ is zero-dimensional, and in that case the minimum is $F(K)$.
\end{theorem}

\begin{proof}
If $A \in \Ws$, then $uA = C(YA)$ by the Stone-Weierstrass theorem. Since $YA$ and $K$ are compact, $C(YA) = C(K)$ implies $YA$ and $K$ are homeomorphic, so $\V(C(K),u) = \{ A \in \Ws : YA = K \}$, and by \ref{VtoS} this has minimum if and only if $K$ is zero-dimensional, and in that case the minimum is $F(K)$. 
\end{proof}

\begin{remark}
This whole project arose from \ref{u}, which is a view of the Stone-Weierstrass situation.
\end{remark}

\begin{theorem}
 Suppose $E \in \Comp$ is infinite and extremally disconnected, and $a(\gamma)$ is an atom of the lattice $\coC$ associated with $E$ as in \ref{a(gamma)}.
 \begin{itemize}
 \item[(i)] $\V(C(E),\mu(a(\gamma)))$ has minimum $F(\dot{E}_{\gamma})$.
 \item[(ii)] If $Y \in \Comp$ and $Y \neq \dot{E}_{\gamma}$ or $E$, then $\V(C(Y),\mu(a(\gamma)))$ has minimum if and only if $Y$ is zero-dimensional, and in that case the minimum is $F(Y)$.
 \end{itemize}
 \end{theorem}
 
 \begin{proof}
$\mS(E,a(\gamma))$ has minimum $\dot{E}_{\gamma}$ by \ref{a(gamma)}(i). Since $\dot{E}_{\gamma}$ is zero-dimensional, statement (i) here follows from \ref{VtoS}.

Similarly, by \ref{a(gamma)}(ii), $\mS(Y,a(\gamma))$ has minimum $Y$, so statement (ii) here follows from \ref{VtoS}.
 \end{proof}
 
 \begin{remark}
 While $a(\gamma)$ is an atom of the lattice $\coC$, the operator $\mu(a(\gamma))$ is not an atom in the lattice $\hoWs$ because the uniform completion operator $u$ satisfies $u \leq \mu(a(\gamma))$ and $u \neq \mu(a(\gamma))$ as 
 \[
 uC(\dot{E}_{\gamma}) = C(\dot{E}_{\gamma}) < C(E) = \mu(a(\gamma))C(E).
 \]
 \end{remark}
 
 \begin{theorem}\label{F(aD)}
 If $E \in \Comp$ is infinite and extremally disconnected, and if $e \in \hoWs$ is the essential completion operator, then $\V(C(E),e)$ has a minimum if and only if $E$ is $\beta D$ for some discrete space $D$, and in that case the minimum is $F(\alpha D)$.
\end{theorem}

\begin{proof}
Let $g \in \coC$ be the Gleason covering operator. Then $\mu(g)=e$. Since $YC(E)=E$, the desired result follows from \ref{c=g} and \ref{VtoS}.
\end{proof}

\begin{theorem}
\begin{itemize}
\item[(i)] $\V(C(\bN),\mu(qF))$ has minimum $F(\aN)$.
\item[(ii)] $\V(C(K),\mu(qF))$ has no minimum if $K \in \Comp$ is a connected almost-$P$ space.
\item[(iii)] $\V(C(K),\mu(qF))$ has minimum $F(K)$ if $K \in \Comp$ is an almost-$P$ space and $K \in \bf{ZD}$.
 \end{itemize}
\end{theorem}

\begin{proof}
(i)--(iii) follow from \ref{c=qF} and \ref{VtoS}.
\end{proof}

\begin{remark}
We have completely answered our ``minimum questions" for $g$ and $\mu(g) = e$ (which is also Dedekind completion in $\bf{W}^*$), and made progress on the questions for $qF$ and $\mu(qF)$ (which is ``$o$-Cauchy completion" in $\bf{W}^*$; see \cite{DHH80}). Since $\coC$ and $\hoWs$ are proper classes, a lot has been left out. Notable perhaps (and between $qF$ and $g$) is the basically disconnected cover and its image under $\mu$ (which is Dedekind $\sigma$-completion in $\bf{W}^*$). Also, there is the projectable hull $p$ in $\bf{W}^*$, which would seem to require different methods since it is not in the range of $\mu$ (see \ref{mustuff}(iii) above).
\end{remark}











\end{document}